\documentclass[runningheads,a4paper]{llncs}

\usepackage{amsmath,amssymb}
\usepackage{cite}

\newcommand{\keywords}[1]{\par\addvspace\baselineskip
\noindent\keywordname\enspace\ignorespaces#1}

\begin{document}

\mainmatter  

\title{Linear and nonlinear fractional Voigt models\thanks{This is 
a preprint of a paper whose final and definite form 
will appear in the Springer LNEE book series.
Submitted 16-April-2016; Revised 11-June-2016; 
Accepted 12-June-2016.}}

\titlerunning{Linear and nonlinear fractional Voigt models}  

\toctitle{Linear and nonlinear fractional Voigt models}


\author{Amar Chidouh\inst{1}\thanks{This work is part of first author's 
Ph.D., which is carried out at Houari Boumediene University, Algeria.}
\and Assia Guezane-Lakoud\inst{2} 
\and Rachid Bebbouchi\inst{3} 
\and \\ Amor Bouaricha\inst{4}
\and Delfim F. M. Torres\inst{5}\thanks{Corresponding author.}}

\authorrunning{A. Chidouh et al.} 

\tocauthor{A. Chidouh, A. Guezane-Lakoud, 
R. Bebbouchi, A. Bouaricha, D. F. M. Torres}


\institute{
Laboratory of Dynamic Systems, 
Houari Boumediene University,
Algiers, Algeria\\
\email{m2ma.chidouh@gmail.com}
\and 
Laboratory of Advanced Materials, 
Badji Mokhtar-Annaba University, Algeria\\
\email{a\_guezane@yahoo.fr}
\and 
Laboratory of Dynamic Systems, 
Houari Boumediene University,
Algiers, Algeria\\
\email{rbebbouchi@hotmail.com}
\and 
Laboratory of Industrial Mechanics, 
Badji Mokhtar-Annaba University, Algeria\\
\email{bouarichaa@yahoo.fr}
\and 
Department of Mathematics, Center for Research 
and Development in Mathematics and Applications (CIDMA),
University of Aveiro, 3810--193 Aveiro, Portugal\\
\email{delfim@ua.pt}}

\maketitle


\begin{abstract}
We consider fractional generalizations of the ordinary 
differential equation that governs the creep phenomenon. 
Precisely, two Caputo fractional Voigt models are considered:
a rheological linear model and a nonlinear one. In the linear
case, an explicit Volterra representation of the solution 
is found, involving the generalized Mittag-Leffler function 
in the kernel. For the nonlinear fractional Voigt model, an 
existence result is obtained through a fixed point theorem.
A nonlinear example, illustrating the obtained existence
result, is given.


\medskip

\end{abstract}


\section{Introduction}
\label{sec:1}

To study the behaviour of viscoelastic materials, one often uses rheological
models that can be of Voigt or Maxwell type or a combination of these basic
models \cite{reol}. For example, the classical phenomenon of creep, in its
simplest form, is known to be governed by a linear ordinary differential 
equation of order one, given by the linear Voigt model:
\begin{equation}
\label{1111}
\eta \frac{d\epsilon (t)}{dt}+E\epsilon (t)=\sigma(t),  
\quad \sigma(0) = 0,
\end{equation}
where $\eta$ is the viscosity coefficient and $E$ is the modulus of the
elasticity. For a given stress history $\sigma$, the solution of (\ref{1111}) 
is given by 
\begin{equation}
\label{simple}
\epsilon(t)=\frac{1}{\eta}
\int\limits_{0}^{t}e^{-\frac{t-s}{\tau}}
\sigma(s)ds,\quad \tau =\frac{\eta}{E},  
\end{equation}
where for $t \leq 0$ the material is
at rest, without stress and strain. The constant $\tau$ is called the
retardation time and has an analogous meaning to relaxation: it is an
estimation of the time required for the creep process to approach
completion. The expression
\begin{equation}
\label{eq:c:f}
k(t)=\frac{1}{E}\left( 1-\exp \left( -\frac{t}{\tau }\right)\right),
\quad t\geq 0,  
\end{equation}
is known as the creep function. In Section~\ref{sec:2} we generalize
\eqref{1111}--\eqref{eq:c:f}.

Fractional calculus has recently become an important tool in the analysis of
viscoelastic phenomena, such as stress-strain relationships in polymeric
materials: in \cite{MR0747787} the connection between the fractional calculus 
and the theory of Abel's integral equation is shown for materials with memory,
while a fractional order Voigt model is proposed in \cite{MR2557002} to better 
simulate the surface wave response of soft tissue-like material phantoms.
For an historical survey of the contributions on the applications 
of fractional calculus in linear viscoelasticty, see \cite{MR2974328}.
In 1996, Mainardi investigated linear fractional relaxation-oscillation 
and fractional diffusion-wave phenomena \cite{[m]}. Several other works in the same
direction of research followed: for an introduction 
to the linear operators of fractional integration 
and fractional differentiation, accessible to applied scientists, 
we refer to \cite{minna}; for a comprehensive overview of fractional 
calculus and waves in linear viscoelastic media see \cite{mina};
for a book devoted to the description of the properties 
of the Mittag-Leffler function, its numerous generalizations 
and their applications in different areas of modern science, 
we refer to \cite{mi23}; for a generalization of the partial 
differential equation of Gaussian diffusion, by using the 
time-fractional derivative, in both the Riemann--Liouville and 
Caputo senses, see \cite{mi1}. Heymans and Podlubny 
have given a physical interpretation of initial conditions 
for fractional differential equations with Riemann--Liouville
fractional derivatives \cite{reo}. Here, motivated by such results, we examine
fractional creep equations involving Caputo derivatives of order $\alpha \in
(0,1)$. Caputo derivatives were chosen because they have a major utility for
treating initial-value problems for physical and engineering applications,
where initial conditions are usually expressed in terms of integer-order
derivatives \cite{MR3384023,MR2721980}. Precisely, we begin by considering 
in Section~\ref{sec:2} the following extension to \eqref{1111}:
\begin{equation}
\label{pb1}
\left\{
\begin{array}{c}
\eta ^{\alpha }\left( ^{C}D_{0}^{\alpha }\epsilon \right)(t)
+E^{\alpha}\epsilon (t)={\sigma }(t),\quad 0<t\leq 1, \\
\epsilon (0)=0,
\end{array}
\right.  
\end{equation}
where $E,\eta >0$ and ${\sigma}$ is a continuous function defined on $[0,1]$.
While the solution (\ref{simple}) of (\ref{1111}) 
is described by an exponential function, we show that
the solution of (\ref{pb1}) is expressed in terms of the Mittag-Leffler
function (see Theorem~\ref{z}), which is a generalization 
of the exponential function and was introduced by Mittag-Leffler 
in \cite{23}, where he investigated some of their properties. 
The Mittag-Leffler function $E_{\alpha }(t)$ with $\alpha>0$ 
is defined by the series representation
\begin{equation}
\label{***}
E_{\alpha }(t)=\sum\limits_{n=0}^{\infty }\frac{t^{n}}{\Gamma (\alpha n+1)},
\quad \alpha >0,\quad t\in \mathbb{C},  
\end{equation}
where $\Gamma$ denotes the Gamma function, valid in the whole complex plane.
A straightforward generalization of the Mittag-Leffler function (\ref{***}), 
due to Wiman \cite{wiman} and used here, is obtained by replacing 
the additive constant $1$ in the argument of the Gamma function 
in (\ref{***}) by an arbitrary complex parameter $\beta$:
\begin{equation}
\label{****}
E_{\alpha ,\beta }(t)=\sum\limits_{n=0}^{\infty }\frac{t^{n}}{\Gamma (\alpha
n+\beta )},\quad \alpha >0,\quad \beta >0,\quad t\in \mathbb{C}.
\end{equation}
Mittag-Leffler functions are considered to be the queen functions of
fractional calculus and they play a fundamental role in the solution
to (\ref{pb1}). Details about the Mittag-Leffler function 
and their importance when solving fractional differential equations 
can be found in \cite{bod,MR2649238,MR3355190}
and references therein. Here we transform 
(\ref{pb1}) as a Volterra integral equation 
to obtain an explicit solution involving
the Mittag-Leffler function (see proof of Theorem~\ref{z}). 
Moreover, we give a physical interpretation to the fractional 
order Voigt model (\ref{pb1}) as a creep phenomenon, 
by finding the corresponding creep function (Theorem~\ref{thm:1}). 
Under some assumptions on $\sigma$, 
when it depends on $\epsilon$ (nonlinear Voigt model), 
in Section~\ref{sec:3} we address 
the question of existence of positive solutions, 
which also contributes to the physical 
interpretation of the model (Theorem~\ref{result}).
Roughly speaking, the existence of nontrivial positive solutions is
obtained by means of the Guo-Krasnosel'skii fixed point theorem. 
We end with an illustrative example and Section~\ref{sec:4} of conclusions.


\section{Solution to the fractional rheological linear Voigt model}
\label{sec:2}

Viscoelastic relations may be expressed in both integral and differential
forms. Differential forms are related to rheological models and provide 
a more direct physical interpretation of the viscoelastic behavior. 
Integral forms are very general and appropriate for theoretical work.
In Section~\ref{sec:1} we introduced the fractional
Voigt model (\ref{pb1}) and explained its physical relevance.
Here we make use of the corresponding integral representation
to obtain an explicit solution to (\ref{pb1}).

\begin{theorem}[The fractional strain]
\label{z}
Assume that the given stress history ${\sigma}$
of the fractional initial value problem (\ref{pb1})
is a continuous function on $[0,1]$. Then 
\begin{equation}
\label{strain}
\epsilon(t)=\frac{1}{\eta ^{\alpha }}
\int_{0}^{t}(t-s)^{\alpha -1}E_{\alpha,\alpha }\left( 
-\left( \frac{t-s}{\tau }\right)^{\alpha }\right) {\sigma}(s)ds,
\end{equation}
$0\leq t\leq 1$, is the fractional strain, that is,
is the solution to (\ref{pb1}).
\end{theorem}

\begin{proof}
Since ${\sigma}$ is a continuous function on $[0,1]$, then
we know from \cite[Theorem~3.24]{kilbass} that 
the fractional initial problem (\ref{pb1})
is equivalent to the Volterra integral equation 
of second kind 
\begin{equation*}
\epsilon (t)=\frac{1}{\eta ^{\alpha }\Gamma (\alpha )}
\int\limits_{0}^{t}(t-s)^{\alpha -1}\sigma (s)ds-\frac{1}{\tau ^{\alpha }
\Gamma(\alpha )}\int\limits_{0}^{t}(t-s)^{\alpha -1}\epsilon (s)ds.
\end{equation*}
To solve this integral equation, we apply the method of successive
approximations. Let us consider the sequence defined by the following
recurrence relation:
$\displaystyle \epsilon_{m}=\frac{I^{\alpha}\sigma}{\eta^{\alpha}} 
-\frac{I^{\alpha }\epsilon _{m-1}}{\tau^{\alpha}}$,
where
$\displaystyle I^{\alpha }z(t)=\frac{1}{\Gamma(\alpha)}
\int\limits_{0}^{t}(t-s)^{\alpha-1}z(s)ds$.
Setting $\displaystyle \epsilon_{0}=\frac{I^{\alpha }\sigma}{\eta ^{\alpha }}$, 
we get
$\displaystyle \epsilon_{1}=\frac{I^{\alpha}\sigma}{\eta^{\alpha}} 
-\frac{I^{2\alpha}\sigma}{\eta^{\alpha}\tau^{\alpha}}$
and
$\displaystyle \epsilon_{2}=\frac{I^{\alpha}\sigma}{\eta^{\alpha}} 
-\frac{I^{2\alpha}\sigma}{\eta^{\alpha }\tau^{\alpha }}
+\frac{I^{3\alpha}\sigma}{\eta^{\alpha}\tau^{2\alpha}}$.
Continuing this process, we obtain that
\begin{equation*}
\epsilon_{m} =\frac{1}{\eta ^{\alpha }}\sum\limits_{k=0}^{m}\left( 
-\frac{1}{\tau ^{\alpha }}\right) ^{k}I^{k\alpha +\alpha }\sigma.
\end{equation*}
Consequently, we have
\begin{equation*}
\epsilon _{m}(t)=\frac{1}{\eta^{\alpha }}\int\limits_{0}^{t}(t-s)^{\alpha
-1}\sum\limits_{k=0}^{m}\frac{(t-s)^{k\alpha }}{\Gamma (k\alpha +\alpha )}
\left(-\frac{1}{\tau ^{\alpha }}\right)^{k}\sigma (s)ds.
\end{equation*}
Taking the limit as $m\rightarrow \infty$, and by (\ref{****}), 
we obtain the explicit solution (\ref{strain}).
\hfill $\qed$
\end{proof}

\begin{remark}
Our fractional problem (\ref{pb1}) provides a generalization 
to the linear Voigt creep model (\ref{1111}).
If we take $\alpha =1$, then 
Theorem~\ref{z} gives the solution (\ref{simple})
to the classical problem (\ref{1111}).
\end{remark}

We now generalize the creep function \eqref{eq:c:f} 
to our fractional Voigt model (\ref{pb1}).

\begin{theorem}[The fractional creep function]
\label{thm:1}
The creep function associated with the fractional 
initial value problem (\ref{pb1}) is given by
\begin{equation}
\label{formula}
k_{\alpha }(t)
=-\left( \frac{\tau }{\eta }\right) ^{\alpha }\left( E_{\alpha}\left( 
-\left( \frac{t}{\tau }\right) ^{\alpha }\right) -1\right).
\end{equation}
\end{theorem}

\begin{proof}
We find the creep function $k_{\alpha}$ by using  
(\ref{strain}), where the latter is defined as
\begin{equation*}
\epsilon (t)=\int_{0}^{t}k_{\alpha }(t-s)d{\sigma }(s),
\quad 0\leq t\leq 1.
\end{equation*}
Integrating expression (\ref{strain}) by parts, we obtain that
\begin{multline*}
\epsilon(t)=\frac{1}{\eta ^{\alpha }}t^{\alpha }E_{\alpha ,\alpha +1}\left(
-\left( \frac{t}{\tau }\right) ^{\alpha }\right) {\sigma }(0)\\
+\frac{1}{\eta^{\alpha }}\int_{0}^{t}(t-s)^{\alpha }
E_{\alpha ,\alpha +1}\left(-\left(
\frac{t-s}{\tau }\right)^{\alpha }\right) {\sigma }^{\prime }(s)ds,
\quad 0\leq t\leq 1.
\end{multline*}
The strain is linear in the stress. Therefore,
the creep function is given by
\begin{equation}
\label{form}
k_{\alpha }(t)=\frac{1}{\eta ^{\alpha }}t^{\alpha }E_{\alpha ,\alpha
+1}\left( -\left( \frac{t}{\tau }\right) ^{\alpha }\right).
\end{equation}
Now, by using the definition of Mittag-Leffler function in (\ref{form}),
we obtain that
\begin{equation*}
\begin{split}
k_{\alpha }(t) 
&=\frac{1}{\eta ^{\alpha }}t^{\alpha}
\sum\limits_{n=0}^{\infty }(-1)^{n}\frac{(\frac{t}{\tau})^{\alpha n}}{
\Gamma (\alpha n+\alpha +1)} 
=\left( \frac{\tau }{\eta }\right) ^{\alpha }\sum\limits_{n=0}^{\infty}
(-1)^{n}\frac{\left( \frac{t}{\tau }\right) ^{\alpha n+\alpha }}{\Gamma
(\alpha n+\alpha +1)} \\
&=-\left( \frac{\tau }{\eta }\right) ^{\alpha }\sum\limits_{n=1}^{\infty}
(-1)^{n}\frac{\left( \frac{t}{\tau }\right) ^{\alpha n}}{\Gamma (\alpha n+1)}
=-\left( \frac{\tau }{\eta }\right) ^{\alpha }\left( E_{\alpha }\left(
-\left( \frac{t}{\tau }\right) ^{\alpha }\right) -1\right).
\end{split}
\end{equation*}
The proof is complete. \hfill $\qed$
\end{proof}

\begin{remark}
If we take $\alpha =1$, then we obtain from 
Theorem~\ref{thm:1} that
\begin{equation*}
k_{1}(t)=\frac{\tau }{\eta }\left( 1-E_{1}\left( 
-\left( \frac{t}{\tau }\right) \right) \right) 
=\frac{1}{E}\left( 1-\exp \left( -\left( \frac{t}{\tau }\right) \right) \right)
=k(t),
\end{equation*}
that is, the creep function (\ref{eq:c:f}) is a special case 
of the $k_{\alpha }(t)$  given by (\ref{formula}).
\end{remark}

Next we generalize (\ref{pb1}) to the nonlinear case, 
where the stress $\sigma$ depends on the strain $\epsilon$.


\section{A nonlinear fractional Voigt model}
\label{sec:3}

By applying the method of successive approximations, we have proved 
in Section~\ref{sec:2} that the fractional initial value problem 
(\ref{pb1}) has a solution $\epsilon$ in $C[0,1]$ given by 
(\ref{strain}). Let us now consider (\ref{pb1}) as a nonlinear problem,
that is, consider a fractional Voigt model described by a 
differential equation with a nonlinear right-hand side $\sigma$ 
depending on $\epsilon$:
\begin{equation}
\label{pp}
\left\{
\begin{array}{c}
\eta ^{\alpha }\left( ^{C}D_{0}^{\alpha }\epsilon \right) (t)+E^{\alpha}
\epsilon (t)={\sigma }(\epsilon (t)),
\quad 0<t\leq 1, \\
\epsilon (0)=0.
\end{array}
\right.  
\end{equation}
We deal with the solvability of the initial value problem (\ref{pp}).
Precisely, we are interested in proving the existence of positive solutions, 
which are the ones that make sense in physics. To establish existence 
of solutions has been a very active research area in mathematics. 
This is particularly true with respect to existence of solutions 
for fractional differential equations \cite{MR3381102},
which is also explained by the development of other fields 
of research, such as physics, mechanics and biology 
\cite{MR3384023,MR3376810,MR2946842}. Many methods 
are used to prove existence of a solution, such 
as the fixed point technique, for which several theories 
are available \cite{MR2962045,MyID:328,MyID:323}. 
In recent years, there has been many papers
investigating the existence of positive solutions: see 
\cite{MyID:345,vanishing1,existence,MR2963223,[9],vanishing2} 
and references therein.  In this section, motivated by many papers 
that discuss the existence of solutions to initial value problems, 
e.g. \cite{amr1,amr2,guezane}, we focus
on the fractional initial value problem (\ref{pp}). 

\begin{theorem}[Existence of a positive solution 
to the nonlinear fractional Voigt model (\ref{pp})]
\label{result} 
Assume that $\sigma: \mathbb{R}_{+}\rightarrow \mathbb{R}_{+}$ 
is a continuous, convex and decreasing function.
Let $\mathbf{E}_{0}:=\lim_{\epsilon \rightarrow 0}
\frac{{\sigma }(\epsilon)}{\epsilon}$ and 
$\mathbf{E}_{\infty }:=\lim_{\epsilon \rightarrow \infty}
\frac{{\sigma }(\epsilon)}{\epsilon}$.
If $\mathbf{E}_{0}=\infty $ and $\mathbf{E}_{\infty }=0$, then problem 
(\ref{pp}) has at least one nontrivial positive bounded solution 
$\epsilon \in X$.
\end{theorem}

Our Theorem~\ref{result} is proved by the Guo--Krasnosel'skii 
fixed point theorem \cite{geo,MR2345924}. Roughly speaking, 
our analysis is mainly based  on the following result on the monotonicity 
of the Mittag-Leffler function, which was first proved by Schneider 
in \cite{schneider}: the generalized Mittag-Leffler function 
$E_{\alpha ,\beta }(-t)$ with $t\geq 0$ is completely monotonic if and only if
$0<\alpha \leq 1$ and $\beta \geq \alpha $. Thus,
if $0<\alpha \leq 1$ and $\beta \geq \alpha$, then
$(-1)^{n}\frac{d^{n}}{dt^{n}}E_{\alpha ,\beta }(-t)\geq 0$
for all $n=0,1,2,\ldots$
Note that $\sigma (0)\neq 0$ because our function $\sigma$ 
is positive and decreasing and we are interested 
in a nontrivial solution.


\subsection{Auxiliary results}

Let $X=C[0,1]$ be the Banach space
of all continuous real functions defined on $[0,1]$ with the norm
$\left\Vert u\right\Vert =\sup_{t\in \lbrack 0,1]}\left\vert u(t)\right\vert$.
Define the operator $T:X\rightarrow X$ by
\begin{equation}
\label{3}
T\epsilon (t)=\frac{1}{\eta ^{\alpha }}\int_{0}^{t}(t-s)^{\alpha-1}
E_{\alpha ,\alpha }\left( -\left( \frac{t-s}{\tau}\right)^{\alpha}
\right){\sigma}(\epsilon (s))ds,\quad 0\leq t\leq 1.  
\end{equation}
Using the Guo--Krasnosel'skii fixed point theorem, we prove existence of
nontrivial positive solutions. For that we first present 
and prove several lemmas. In what follows, $K$ is the cone
$K:=\left\{\epsilon \in X:\ \epsilon (t)\geq 0, \ 0\leq t\leq 1\right\}$.

\begin{lemma}
\label{123}The operator $T:K\rightarrow K$\ is completly continuous.
\end{lemma}

\begin{proof}
Taking into account the monotonicity of the Mittag-Leffler function, we have
that the operator $T:K\rightarrow K$ is continuous in view of the
assumptions of nonnegativeness and continuity of ${\sigma }$.
Let $B\subset K$ be the bounded set $B:=B(0,\eta _{0})=\left\{ \epsilon \in
K:\left\Vert \epsilon \right\Vert \leq \eta _{0},\ \eta _{0}>0\right\} $,
and let $\rho =\max_{0\leq t\leq 1,0\leq \epsilon 
\leq \eta _{0}}{\sigma}(\epsilon (t))+1$. 
Then, for any $\epsilon \in B$, we have
\begin{equation*}
\begin{split}
\left\vert T\epsilon (t)\right\vert 
&=\left\vert \frac{1}{\eta^{\alpha}}
\int_{0}^{t}(t-s)^{\alpha -1}E_{\alpha ,\alpha }\left( -\left( 
\frac{t-s}{\tau }\right) ^{\alpha }\right) {\sigma }(\epsilon (s))ds\right\vert\\
&\leq \frac{1}{\eta ^{\alpha }}\int_{0}^{t}(t-s)^{\alpha -1}E_{\alpha,\alpha}
\left(-\left( \frac{t-s}{\tau }\right) ^{\alpha }\right)
\left\vert {\sigma }(\epsilon (s))\right\vert ds \\
&\leq \frac{\rho }{\eta ^{\alpha }\Gamma (\alpha +1)}t^{\alpha}
\Rightarrow \left\Vert T\epsilon \right\Vert 
\leq \frac{\rho }{\eta^{\alpha }\Gamma (\alpha +1)}.
\end{split}
\end{equation*}
Hence, $T(B)$ is uniformly bounded. Now, we prove that the operator $T$ is
equicontinuous for each $\epsilon \in B$, any $\varepsilon >0$, and 
$t_{1},t_{2}\in \lbrack 0,1]$ with $t_{2}>t_{1}$. Let $\delta =\left( 
\frac{\eta ^{\alpha }\Gamma (\alpha +1)\varepsilon }{2
\rho}\right)^{\frac{1}{\alpha }}$. Then, for 
$\left\vert t_{2}-t_{1}\right\vert <\delta$,
\begin{equation*}
\begin{split}
\vert T\epsilon (t_{1}) &- T\epsilon (t_{2})\vert \\
&\leq \frac{\rho}{\eta ^{\alpha }\Gamma (\alpha )}\left( 
\int_{0}^{t_{1}}((t_{1}-s)^{\alpha-1}-(t_{2}-s)^{\alpha -1})ds
+\int_{t_{1}}^{t_{2}}(t_{2}-s)^{\alpha-1}ds\right)\\
&\leq \frac{\rho\left((t_{1}^{\alpha}+(t_{2}-\ t_{1})^{\alpha }
-t_{2}^{\alpha }+(t_{2}-\ t_{1})^{\alpha }\right)}{\eta^{\alpha}
\Gamma(\alpha+1)}
\leq \frac{2\rho(t_{2}-t_{1})^{\alpha}}{\eta^{\alpha}
\Gamma(\alpha +1)}=\varepsilon.
\end{split}
\end{equation*}
Therefore, $T(B)$ is equicontinuous. From the Arzela--Ascoli theorem, 
it follows that operator $T$ is completely continuous.
\hfill $\qed$
\end{proof}

The following results are also used in the proof 
of our Theorem~\ref{result}.

\begin{lemma}[Jensen's inequality \cite{MR924157}]
\label{1'} 
Let $\mu$ be a positive measure and let $\Omega$ be a
measurable set with $\mu (\Omega )=1$. Let $I$ be an interval and suppose
that $u$ is a real function in $L^{1}(\Omega )$ with $u(t)\in I$ for all 
$t\in \Omega$. If $f$ is convex on $I$, then
\begin{equation*}
f\left(\int_{\Omega }u(t)d\mu(t)\right) 
\leq \int_{\Omega }(f \circ u)(t)d\mu (t).
\end{equation*}
\end{lemma}

\begin{lemma}[Guo--Krasnosel'skii's fixed point theorem \protect\cite{geo}]
\label{kras_copy(1)} 
Let $X$ be a Banach space and let $K\subset X$ be a cone. 
Assume $\Omega _{1}$ and $\Omega_{2}$ are bounded open subsets of $X$
with $0\in \Omega _{1}\subset \overline{\Omega }_{1}\subset $ $\Omega _{2}$,
and let $T:K\cap (\overline{\Omega }_{2}\backslash \Omega_{1})\rightarrow K$
be a completely continuous operator such that either
\begin{enumerate}
\item $Tu\leq u$ for any $u\in K\cap \partial \Omega _{1}$ and $Tu\geq u$
for any $u\in K\cap \partial \Omega _{2}$, or

\item $Tu\geq u$ for any $u\in K\cap \partial \Omega _{1}$ and $Tu\leq u$
for any $u\in K\cap \partial \Omega _{2}$.
\end{enumerate}
Then $T$ has a fixed point in $K\cap (\overline{\Omega }_{2}\backslash
\Omega _{1})$.
\end{lemma}

Since ${\sigma}$ is continuous on $\mathbb{R}_{+}$, 
we can define the function
$\displaystyle \overline{{\sigma }}(\epsilon)
=\max_{0\leq z\leq \epsilon }\left\{{\sigma}(z)\right\}$.
Let
$\displaystyle \overline{\mathbf{E}}_{0}=\lim_{\epsilon \rightarrow 0}
\frac{\overline{{\sigma}}(\epsilon)}{\epsilon}$
and
$\displaystyle \overline{\mathbf{E}}_{\infty}
=\lim_{\epsilon \rightarrow \infty }
\frac{\overline{{\sigma}}(\epsilon)}{\epsilon}$.

\begin{lemma}[{See \protect\cite{[9]}}]
\label{1} 
Assume ${\sigma}$ is continuous. Then $\overline{\mathbf{E}}_{0}
=\mathbf{E}_{0}$ and $\overline{\mathbf{E}}_{\infty }
=\mathbf{E}_{\infty}$.
\end{lemma}

We are now in condition to prove Theorem~\ref{result}.


\subsection{Proof of Theorem~\ref{result}}

By Lemma~\ref{123}, we know that the operator (\ref{3}) 
is completely continuous. Now, using Lemma~\ref{kras_copy(1)}, 
we give a proof to our result. Denote
$\Omega _{r_{i}}=\{\epsilon \in X:\left\Vert \epsilon \right\Vert <r_{i}\}$.
When $\mathbf{E}_{0}=\infty$, we can choose $r_{1}>0$ sufficiently small
such that ${\sigma }(\epsilon )\geq \varpi \epsilon$ for 
$\epsilon \leq r_{1}$, where $\varpi $ satisfies
$\displaystyle \left( \varpi \frac{E_{\alpha ,\alpha }(
-\frac{1}{\tau ^{\alpha }})}{\eta^{\alpha }\alpha (\alpha +1)}\right) >1$.
Now let us show that $T\epsilon \leq \epsilon$ for any $\epsilon 
\in K\cap\partial \Omega _{r_{1}}$. In fact, if there exists 
$\epsilon _{1}\in \partial \Omega _{r_{1}}$ such that 
$T\epsilon _{1}\leq \epsilon _{1}$, the following inequalities hold:
\begin{eqnarray*}
\left\Vert \epsilon _{1}\right\Vert &\geq &\left\Vert T\epsilon_{1}
\right\Vert \geq \int_{0}^{1}T\epsilon _{1}(t)dt \\
&\geq &\frac{1}{\eta ^{\alpha }}\int_{0}^{1}\int_{0}^{t}(t-s)^{\alpha-1}
E_{\alpha ,\alpha }\left( -\left( \frac{t-s}{\tau }\right) ^{\alpha}\right) 
{\sigma }(\epsilon _{1}(s))ds dt\\
&\geq &\frac{1}{\eta ^{\alpha }}E_{\alpha ,\alpha }\left( 
-\frac{1}{\tau^{\alpha }}\right) 
\int_{0}^{1}{\sigma}(\epsilon_{1}(s))\left(
\int_{s}^{1}(t-s)^{\alpha -1}dt\right) ds \\
&\geq &\frac{E_{\alpha ,\alpha }\left( -\frac{1}{\tau ^{\alpha }}\right) }{
\eta ^{\alpha }\alpha }\int_{0}^{1}(1-s)^{\alpha }{\sigma }(\epsilon_{1}(s))ds \\
&\geq &\frac{E_{\alpha ,\alpha }\left( 
-\frac{1}{\tau ^{\alpha }}\right) }{\eta^{\alpha }\alpha(\alpha +1)}
\int_{0}^{1}(\alpha +1)(1-s)^{\alpha }{\sigma}(\epsilon _{1}(s))ds.
\end{eqnarray*}
Then, by Lemma~\ref{1'}, we have
\begin{equation*}
\begin{split}
\left\Vert \epsilon _{1}\right\Vert 
&\geq \frac{E_{\alpha ,\alpha }(
-\frac{1}{\tau^{\alpha }})}{\eta ^{\alpha }\alpha (\alpha +1)}{\sigma}\left(
\int_{0}^{1}(\alpha +1)(1-s)^{\alpha }\epsilon _{1}(s)ds\right) \\
&\geq \frac{E_{\alpha ,\alpha }(-\frac{1}{\tau ^{\alpha }})}{\eta^{\alpha}
\alpha(\alpha +1)}{\sigma }\left( \int_{0}^{1}(\alpha +1)
(1-s)^{\alpha}r_{1}ds\right) \\
&\geq \frac{E_{\alpha ,\alpha }(
-\frac{1}{\tau^{\alpha}})}{\eta^{\alpha}\alpha (\alpha 
+1)}{\sigma }\left( r_{1}\right)
\geq \varpi \frac{E_{\alpha ,\alpha }(
-\frac{1}{\tau^{\alpha }})}{\eta^{\alpha }\alpha (\alpha +1)}r_{1} > r_{1},
\end{split}
\end{equation*}
which is a contradiction. Since $\mathbf{E}_{\infty }=0$, Lemma~\ref{1}
implies $\overline{\mathbf{E}}_{\infty }=0$. Thus, there exists 
$r_{2}\in (r_{1},\infty )$ such that
$\overline{{\sigma }}(r_{2})<\eta ^{\alpha }\Gamma (\alpha +1)r_{2}$.
Note that $0<\Gamma (\alpha +1)<1$ for all $\alpha \in (0,1)$. We now show
that $T\epsilon \geq \epsilon $ for any $\epsilon \in K
\cap \partial \Omega_{r_{2}}$. If there exists 
$\epsilon_{2}\in \partial \Omega _{r_{2}}$ such that 
$T\epsilon_{2}\geq \epsilon_{2}$, then
\begin{eqnarray*}
\left\Vert \epsilon _{2}\right\Vert &\leq &\left\Vert T\epsilon_{2}
\right\Vert =\sup_{t\in \lbrack 0,1]}\frac{1}{\eta ^{\alpha }}
\int_{0}^{t}(t-s)^{\alpha -1}E_{\alpha ,\alpha }\left( 
-\left(\frac{t-s}{\tau }\right)^{\alpha }\right) {\sigma }\left( 
\epsilon_{2}(s)\right) ds \\
&\leq &\frac{1}{\eta ^{\alpha }\Gamma (\alpha +1)}\max_{0<\epsilon_{2}
<r_{2}}{\sigma}(\epsilon_{2}) \\
&\leq &\frac{1}{\eta ^{\alpha }\Gamma (\alpha +1)}
\overline{{\sigma}}(r_{2})<r_{2},
\end{eqnarray*}
which is a contradiction. Hence, from the first part of the 
Lemma~\ref{kras_copy(1)}, $T$ has a fixed point in 
$K\cap(\overline{\Omega }_{r_{2}}\backslash \Omega _{r_{1}})$. 
Therefore, problem (\ref{pp}) has at least one nontrivial 
bounded positive solution $\epsilon \in X$.


\subsection{An example}

We now take a simple example to illustrate our analysis.
Consider problem
\begin{equation}
\label{exe}
\left\{
\begin{array}{c}
^{C}D_{0}^{\frac{1}{2}}\epsilon(t) +\sqrt{2} \epsilon(t) 
=\frac{1}{1+\epsilon(t)},
\quad 0<t\leq 1, \\
\epsilon(0)=0.
\end{array}
\right.  
\end{equation}
As already mentioned, the term $\frac{1}{1+\epsilon(t)}$
is the constitutive equation of the creep.
Function ${\sigma }(\epsilon )=\frac{1}{1+\epsilon }
: \mathbb{R}_{+} \rightarrow \mathbb{R}_{+}$  
is continuous, convex and decreasing with $\sigma(0)\neq 0$.
Due to the fact that $\mathbf{E}_{0}=\infty $ and $\mathbf{E}_{\infty }=0$,
it follows from Theorem~\ref{result} that (\ref{exe}) has at least
one nontrivial bounded positive solution $\epsilon \in C[0,1]$.


\section{Conclusion}
\label{sec:4}

In this work we investigated the creep phenomenon described by 
linear and nonlinear fractional order Voigt models 
involving the Caputo derivative. We were able to give 
an integral representation of our initial value problem 
and to compute the creep function in the linear case. 
The obtained Volterra integral equation involves 
the Mittag-Leffler function in the kernel, which
is a completely monotonic function in the context of our considerations.
This property was the key of our analysis to establish existence 
of positive solutions.


\section*{Acknowledgments}

This research was finished while Chidouh was visiting 
University of Aveiro, Portugal. The hospitality of the 
host institution and the financial support of Houari 
Boumedienne University, Algeria, are here gratefully acknowledged. 
Torres was supported by CIDMA and FCT within project UID/MAT/04106/2013.
The authors are grateful to two referees for valuable
comments and suggestions.




\begin{thebibliography}{99}

\bibitem{MR3384023} 
M. I. Abbas, 
Existence and uniqueness of solution for a boundary value problem 
of fractional order involving two Caputo's fractional derivatives, 
Adv. Difference Equ. \textbf{2015}, 2015:252, 19~pp.

\bibitem{MR3381102}
S. Abbas\ and\ M. Benchohra, 
{\it Advanced functional evolution equations and inclusions}, 
Developments in Mathematics, 39, Springer, Cham, 2015. 

\bibitem{MR2962045} 
S. Abbas, M. Benchohra\ and\ G. M. N'Gu\'er\'ekata,
\textit{Topics in fractional differential equations}, 
Developments in Mathematics, 27, Springer, New York, 2012.

\bibitem{MyID:328} 
N. Benkhettou, A. Hammoudi\ and\ D. F. M. Torres,
Existence and uniqueness of solution for a fractional 
Riemann--Liouville initial value problem on time scales, 
J. King Saud Univ. Sci. 28 (2016), no.~1, 87--92.
{\tt arXiv:1508.00754}

\bibitem{amr1} 
A. Chidouh, A. Guezane-Lakoud\ and\ R. Bebbouchi, 
Positive solutions for an oscillator fractional initial value problem, 
J. Appl. Math. Comput., in press. DOI:10.1007/s12190-016-0996-9

\bibitem{amr2} 
A. Chidouh, A. Guezane-Lakoud\ and\ R. Bebbouchi,
Positive solutions of the fractional relaxation 
equation using lower and upper solutions, 
Vietnam J. Math., in press. DOI:10.1007/s10013-016-0192-0

\bibitem{MyID:345}
A. Chidouh\ and\ D. F. M. Torres,
A generalized Lyapunov's inequality for a fractional boundary value problem,
J. Comput. Appl. Math., in press. DOI:10.1016/j.cam.2016.03.035
{\tt arXiv:1604.00671}

\bibitem{MyID:323} 
A. Debbouche, J. J. Nieto\ and\ D. F. M. Torres, 
Optimal solutions to relaxation in multiple control problems 
of Sobolev type with nonlocal nonlinear fractional differential equations, 
J. Optim. Theory Appl., in press. DOI:10.1007/s10957-015-0743-7
{\tt arXiv:1504.05153}

\bibitem{mi23} 
R. Gorenflo, A. A. Kilbas, F. Mainardi\ and\ S. V. Rogosin,
\textit{Mittag-Leffler functions, related topics and applications}, 
Springer Monographs in Mathematics, Springer, Heidelberg, 2014.

\bibitem{minna} 
R. Gorenflo\ and\ F. Mainardi, 
Fractional calculus: integral and differential equations of fractional order, 
in \textit{Fractals and fractional calculus in continuum mechanics (Udine, 1996)},
223--276, CISM Courses and Lectures, 378, Springer, Vienna, 1997.

\bibitem{vanishing1} 
J. R. Graef, L. Kong\ and\ H. Wang, 
A periodic boundary value problem with vanishing Green's function, 
Appl. Math. Lett. \textbf{21} (2008), no.~2, 176--180.

\bibitem{guezane} 
A. Guezane-Lakoud, 
Initial value problem of fractional order, 
Cogent Mathematics \textbf{2} (2015), no.~1, Art.~ID 1004797.

\bibitem{geo} 
D. J. Guo\ and\ V. Lakshmikantham, 
\textit{Nonlinear problems in abstract cones}, 
Notes and Reports in Mathematics in Science and Engineering, 5, 
Academic Press, Boston, MA, 1988.

\bibitem{reo} 
N. Heymans\ and\ I. Podlubny, 
Physical interpretation of initial conditions for fractional 
differential equations with Riemann-Liouville fractional derivatives, 
Rheologica Acta 45 (2006), no.~5, 765--771.

\bibitem{MR3376810} 
J. Jiang, C. F. Li, D. Cao\ and\ H. Chen, 
Existence and uniqueness of solution for fractional differential 
equation with causal operators in Banach spaces, 
Mediterr. J. Math. \textbf{12} (2015), no.~3, 751--769.

\bibitem{kilbass} 
A. A. Kilbas, H. M. Srivastava\ and\ J. J. Trujillo,
\textit{Theory and applications of fractional differential equations},
North-Holland Mathematics Studies, 204, Elsevier, Amsterdam, 2006.

\bibitem{MR0747787} 
R. C. Koeller, 
Applications of fractional calculus to the theory of viscoelasticity, 
Trans. ASME J. Appl. Mech. \textbf{51} (1984), no.~2, 299--307.

\bibitem{existence} 
N. Li\ and\ C. Wang, 
New existence results of positive solution 
for a class of nonlinear fractional differential equations, 
Acta Math. Sci. Ser. B Engl. Ed. \textbf{33} (2013), no.~3, 847--854.

\bibitem{[m]} 
F. Mainardi, 
Fractional relaxation-oscillation and fractional diffusion-wave phenomena, 
Chaos Solitons Fractals \textbf{7} (1996), no.~9, 1461--1477.

\bibitem{mina} 
F. Mainardi, 
\textit{Fractional calculus and waves in linear viscoelasticity}, 
Imp. Coll. Press, London, 2010.

\bibitem{MR2974328} 
F. Mainardi, 
An historical perspective on fractional calculus in linear viscoelasticity, 
Fract. Calc. Appl. Anal. \textbf{15} (2012), no.~4, 712--717.

\bibitem{mi1} 
F.~Mainardi, A.~Mura, G.~Pagnini\ and\ R.~Gorenflo,
Time-fractional diffusion of distributed order, 
J. Vib. Control \textbf{14} (2008), no.~9-10, 1267--1290.

\bibitem{reol} 
S. P. C. Marques\ and\ G. J. Creus, 
\textit{Computational viscoelasticity}, 
Springer Briefs in Applied Sciences and Technology,
Springer, Heidelberg, 2012.

\bibitem{MR2557002} 
F. C. Meral, T. J. Royston\ and\ R. Magin, 
Fractional calculus in viscoelasticity: an experimental study, 
Commun. Nonlinear Sci. Numer. Simul. \textbf{15} (2010), no.~4, 939--945.

\bibitem{23} 
G. Mittag-Leffler, 
Sur la repr\'esentation analytique 
d'une branche uniforme d'une fonction monog\`ene, 
Acta Math. \textbf{29} (1905), no.~1, 101--181.

\bibitem{MR2721980} 
D. Mozyrska\ and\ D. F. M. Torres, 
Minimal modified energy control for fractional 
linear control systems with the Caputo derivative, 
Carpathian J. Math. \textbf{26} (2010), no.~2, 210--221.
{\tt arXiv:1004.3113}

\bibitem{bod} 
I. Podlubny, 
\textit{Fractional differential equations},
Mathematics in Science and Engineering, 198, 
Academic Press, San Diego, CA, 1999.

\bibitem{MR924157}
W. Rudin, 
{\it Real and complex analysis}, 
third edition, McGraw-Hill, New York, 1987.

\bibitem{schneider} 
W. R. Schneider, 
Completely monotone generalized Mittag-Leffler functions, 
Exposition. Math. \textbf{14} (1996), no.~1, 3--16.

\bibitem{MR2946842} 
M. R. Sidi Ammi, E. H. El Kinani\ and\ D. F. M. Torres,
Existence and uniqueness of solutions to functional 
integro-differential fractional equations, 
Electron. J. Differential Equations \textbf{2012} (2012), no.~103, 9~pp.
{\tt arXiv:1206.3996}

\bibitem{MR2345924} 
M. R. Sidi Ammi\ and\ D. F. M. Torres, 
Existence of positive solutions for non local 
$p$-Laplacian thermistor problems on time scales, 
JIPAM. J. Inequal. Pure Appl. Math. \textbf{8} (2007), no.~3, Art.~69, 10~pp.
{\tt arXiv:0709.0415}

\bibitem{MR2963223} 
M. R. Sidi Ammi\ and\ D. F. M. Torres, 
Existence and uniqueness of a positive solution 
to generalized nonlocal thermistor problems 
with fractional-order derivatives, 
Differ. Equ. Appl. \textbf{4} (2012), no.~2, 267--276.
{\tt arXiv:1110.4922}

\bibitem{MR2649238} 
A. L. Soubhia, R. F. Camargo, E. C. de Oliveira\ and\ J. Vaz, 
Theorem for series in three-parameter Mittag-Leffler function, 
Fract. Calc. Appl. Anal. \textbf{13} (2010), no.~1, 9--20.

\bibitem{[9]} 
H. Wang, 
On the number of positive solutions of nonlinear systems, 
J. Math. Anal. Appl. \textbf{281} (2003), no.~1, 287--306.

\bibitem{vanishing2} 
J. R. L. Webb, 
Boundary value problems with vanishing Green's function, 
Commun. Appl. Anal. \textbf{13} (2009), no.~4, 587--595.

\bibitem{wiman} 
A. Wiman, 
\"Uber den Fundamentalsatz in der Teorie der Funktionen $E\sp a(x)$, 
Acta Math. \textbf{29} (1905), no.~1, 191--201.

\bibitem{MR3355190} 
B. Y. Ya\c{s}ar, 
Generalized Mittag-Leffler function and its properties, 
New Trends Math. Sci. \textbf{3} (2015), no.~ 12--18.

\end{thebibliography}
\end{document}